\documentclass[11pt]{amsart}
\usepackage{amsmath, amssymb}
\usepackage{amsfonts}
\usepackage{graphicx}
\usepackage{mathrsfs}
\usepackage[arrow,matrix,curve,cmtip,ps]{xy}
\usepackage{amsthm}
\usepackage{enumerate}
\usepackage{color}
\usepackage[colorlinks]{hyperref}
\usepackage{tikz}

\allowdisplaybreaks

\newtheorem{thm}{Theorem}[section]
\newtheorem{prop}[thm]{Proposition}
\newtheorem{cor}[thm]{Corollary}
\newtheorem{lem}[thm]{Lemma}

\newtheorem*{theorem*}{Theorem}
\theoremstyle{remark}
\newtheorem{rem}[thm]{Remark}
\newtheorem{defn}[thm]{Definition}
\newtheorem{example}[thm]{Example}

\numberwithin{equation}{section}

\newcommand{\Z}{\mathbb{Z}}
\newcommand{\N}{\mathbb{N}}

\newcommand{\La}{\Lambda}
\newcommand{\la}{\lambda}
\newcommand{\kp}{\mathrm{KP}_R(\Lambda)}
\newcommand{\Lamin}{\La^{\mathrm{min}}}
\newcommand{\FE}{\mathrm{FE}(\Lambda)}

\begin{document}
\title[Purely infinite simple Kumjian-Pask algebras]{Purely infinite simple Kumjian-Pask algebras}

\author{Hossein Larki}

\address{Department of Mathematics\\ Faculty of Mathematical Sciences and Computer\\ Shahid Chamran University of Ahvaz, Iran}
\email{h.larki@scu.ac.ir}

%\thanks{This research was supported by NSF Grant DMS-1234567}

\date{\today}



\begin{abstract}
Given any finitely aligned higher-rank graph $\Lambda$ and any unital commutative ring $R$, the Kumjian-Pask algebra $\mathrm{KP}_R(\Lambda)$ is known as the higher-rank generalization of Leavitt path algebras. After characterizing simple Kumjian-Pask algebras by L.O. Clark and Y.E.P. Pangalela (and others), we focus  in this article on the purely infinite simple ones. Briefly, we show that if $\mathrm{KP}_R(\Lambda)$ is simple and every vertex of $\Lambda$ is reached from a generalized cycle with an entrance, then $\mathrm{KP}_R(\Lambda)$ is purely infinite. We next prove a standard dichotomy for simple Kumjian-Pask algebras: in the case that each vertex of $\Lambda$ is reached only from finitely many vertices and $\kp$ is simple, then $\kp$ is either purely infinite or locally matritial. This result covers all unital simple Kumjian-Pask algebras.
\end{abstract}

\maketitle

%%%%%%%%%%%%%%%%%%%%%%%%%%%%%%%%%%%%%%%%%%%%%%%%

\section{Introduction}

Motivated from the work of Roberston and Steger in \cite{rob99}, Kumjian and Pask introduced the notion of a higher-rank graph as analogous of directed graphs and an associated $C^*$-algebra as higher-rank generalization of graph $C^*$-algebras. Directed graphs and their generalizations provide a framework to analysis the associated $C^*$-algebras so that many specific relations and properties of a $C^*$-algebra may be expressed by simple and visible features of the underlying graph. So, the basic problem in the investigation of graph $C^*$-algebras and their generalizations is ``how can realize an structural property of the $C^*$-algebra by observing the underlying graph".

Despite some similarities in definitions, the structure of higher-rank graphs and their $C^*$-algebras are more complicated than that of ordinary directed graphs. However, there has been a great deal of attention to the structure of higher-rank graph $C^*$-algebras (see \cite{rae03,rae04,far05,sim06,rob07,rob09,sho12,lew10,pas06,eva12} for example) because they contain interesting kinds of $C^*$-algebras besides the classical graph $C^*$-algebras such as tensor products of graph $C^*$-algebras \cite{kum00} and many simple A$\mathbb{T}$-algebras with real rank zero \cite{pas06} among others.

In \cite{kum00}, Kumjian and Pask only considered row-finite $k$-graphs with no sources. After that, Raeburn, Sims, and Yeend developed the Kumjian-Pask's construction by associating a $C^*$-algebra to a locally convex row-finite $k$-graph with possible sources \cite{rae03} and to a finitely aligned $k$-graph \cite{rae04} as the most general higher-rank graph $C^*$-algebras.

Let $R$ be a unital commutative ring. Associated to every finitely aligned higher-rank graph (or $k$-graph) $\La$, the Kumjian-Pask algebra $\kp$ is a specific universal $R$-algebra as the algebraic version of higher-rank graph $C^*$-algebras. They were first introduced in \cite{pin13} for row-finite $k$-graphs without sources, and then extended to locally convex row-finite and finitely aligned $k$-graphs \cite{cla14,cla16}. Note that the class of Kumjian-Pask algebras includes all Leavitt path algebras by identifying directed graphs as 1-graphs. However, there are Kumjian-Pask algebras which do not belong to the class of Leavitt path algebras (see \cite[Section 7]{pin13}).

The concept of pure infiniteness was introduced in \cite{ara02} to classify simple rings as an algebraic analogue of that for $C^*$-algebras \cite{cun81}. It was then generalized in \cite{pin10} for possibly non-simple setting. It is well-known from \cite{abr06,lar13} that a simple Leavitt path algebra $L_R(E)$ associated to a directed graph $E$ is purely infinite if and only if every vertex of $E$ is reached from a cycle (the direction of graphs in \cite{abr06,lar13} was considered as contrariwise of ours). So, one may want to have a higher-rank version of this result for the Kumjian-Pask algebras. The aim of present article is to investigate the pure infiniteness of a simple Kumjian-Pask algebra $\kp$ when $\La$ is finitely aligned. We use the notion of \emph{generalized cycles} introduced by Evans in \cite{eva02} to give a (sufficient) condition under which a simple $\kp$ would be purely infinite. In particular, we can determine \emph{all} unital purely infinite simple Kumjian-Pask algebras. Furthermore, we discuss on the ideal structure of $\kp$ when $\La$ is cofinal and aperiodic.

The article is organized as follows. We begin by Section \ref{sec2} with a review of $k$-graphs and associated Kumjian-Pask algebras. In Sections \ref{sec3} and \ref{sec4}, we focus on the cofinality and aperiodicity of $k$-graphs as the fundamental properties for characterizing simple Kumjian-Pask algebras. In particular, we verify the ideal structure of $\kp$ when $\La$ is cofinal and aperiodic. Moreover, in Section \ref{sec4}, we gives some relations between the aperiodicity and generalized cycles of $\La$.

In Sections \ref{sec6} and \ref{sec7}, we investigate the pure infiniteness of simple Kumjian-Pask algebras. Theorem \ref{thm6.4} gives some conditions for $\La$ and $R$ to insure $\kp$ is purely infinite simple. It is the higher-rank analogue of \cite[Theorem 11]{abr06} and \cite[Proposition 3.1]{lar13}. In Section \ref{sec7}, we consider $k$-graphs $\La$ with this property: for each vertex $v$ in $\La$, there are at most finitely many vertices connecting to $v$. In this case, we show a dichotomy for simple Kumjian-Pask algebras $\kp$: if $\La$ contains a cycle, then $\kp$ is purely infinite; otherwise, $\kp$ is locally matricial. Note that this result covers all unital simple Kumjian-pask algebras.

%%%%%%%%%%%%%%%%%%%%%%%%%%%%%%%%%%%%%%%%%%%%%%%%

\section{Preliminaries}\label{sec2}

In this section, we review the basic facts about higher-rank graphs from \cite{kum00,rae03,rae04} and their Kumjian-Pask algebras from \cite{pin13,cla16}.

\subsection{Higher-rank graphs}

Let $\N$ be the set of non-negative integers. Fixed an integer $k\geq 1$, we regard $\N^k$ as a semigroup under pointwise addition and denote the identity $(0,\ldots,0)\in \N^k$ by $0$. We denote by $e_1,\ldots,e_k$ the generators of $\N^k$, where the $i^{\mathrm{th}}$ coordinate of $e_i$ is $1$ and the others are $0$. For $n\in \N^k$, we write $n=(n_1,\ldots,n_k)$ and use $\leq$ for the partial order on $\N^k$ given by $m\leq n$ if $m_i\leq n_i$ for all $i$. We also write $m\vee n$ for the coordinate-wise maximum and $m\wedge n$ for the coordinate-wise minimum.

Following \cite{kum00}, a \emph{higher-rank graph} or \emph{$k$-graph} $\La=(\La^0,\La,r,s)$ is a countable small category $\La$ equipped with a \emph{degree functor} $d:\La \rightarrow \N^k$ satisfying the \emph{unique factorisation property}: if $\la\in\La$ and $d(\la)=m+n$ for $m,n\in\N^k$, then there exist unique $\la_1,\la_2\in \La$ such that $d(\la_1)=m$, $d(\la_2)=n$ and $\la=\la_1 \la_2$. We usually denote $\la(0,m):=\la_1$ and $\la(m,d(\la)):=\la_2$.

Notice that we may view every 1-graph as a directed graph where the degree of each morphism is equal to its length. So, for convenience, we refer to the objects in $\La^0$ as vertices and think of each $\la\in\La$ as a path (of rank $k$) from $s(\la)$ to $r(\la)$. If $\la,\mu\in\La$, then the composition $\la\mu$ makes sense if and only if $r(\mu)=s(\la)$. Recall that we have $\La^0\subseteq \La$ where elements of $\La^0$ are the paths of $\La$ with degree $0$. For $n\in\N^k$, we write $\La^n$ for $d^{-1}(n)=\{\la\in\La:d(\la)=n\}$. Given $\la\in\La$ and $E\subseteq \La$ we define
\begin{align*}
\la E&:=\{\la\mu:\mu\in E, r(\mu)=s(\la)\} ~~\mathrm{and}\\
E\la&:=\{\mu\la:\mu\in E, s(\mu)=r(\la)\}.
\end{align*}
Moreover, if $H\subseteq \La^0$ and $E\subseteq \La$, we write
$$H E:=\{\mu\in E:r(\mu)\in H\} ~~ \mathrm{and} ~~ E H:=\{\mu\in E: s(\mu)\in H\}.$$

We say that $\La$ is \emph{row-finite} if $v\La^n$ is finite for every $v\in \La^0$ and $n\in \N^k$. We also say $\La$ to be \emph{locally convex} if for every $v\in \La^0$, $1\leq i\neq j\leq k$, and every $\la\in v\La^{e_i}$, $\mu\in v\La^{e_j}$, the sets $s(\la)\La^{e_j}$ and $s(\mu)\La^{e_i}$ are nonempty \cite[Definition 3.10]{rae03}.

Given $\mu,\nu\in \La$, a \emph{minimal common extension} for $\mu$ and $\nu$ is a path $\la$ such that
$$d(\la)=d(\mu)\vee d(\nu) ~~ \mathrm{and} ~~ \la=\mu\alpha=\nu\beta ~~\mathrm{for~ some~} \alpha,\beta\in\La.$$
We denote by $\mathrm{MCE}(\mu,\nu)$ the set of all minimal common extension of $\mu$ and $\nu$. We also denote
$$\La^{\mathrm{min}}(\mu,\nu):=\left\{(\alpha,\beta)\in\La\times \La: \mu\alpha=\nu\beta\in \mathrm{MCE}(\mu,\nu)\right\}$$
and if $\mu\in\La$, $E\subseteq \La$, then
$$\mathrm{Ext}(\mu;E):=\bigcup_{\nu\in E}\left\{\alpha:(\alpha,\beta)\in \La^{\mathrm{min}}(\mu,\nu) \mathrm{~for ~some~} \beta\in\La \right\}.$$

\begin{defn}[{\cite[Definition 2.2]{rae03}}]
A $k$-graph $\La$ is called \emph{finitely aligned} if $\La^{\mathrm{min}}(\mu,\nu)$ is finite (possibly empty) for all $\mu,\nu\in\La$.
\end{defn}

Throughout the article, all $k$-graphs will be assumed to be finitely aligned.

\subsection{Kumjian-Pask algebras}

Let $\La$ be a $k$-graph and $v\in \La^0$. A subset $E\subseteq v\La$ is called \emph{exhaustive} if for every $\la\in v\La$, there exists $\mu\in E$ such that $\La^{\mathrm{min}}(\mu,\nu)\neq \emptyset$. Let us denote $\mathrm{FE}(\La)$ the collection of all finite and exhaustive sets in $\La$, that is,
$$\mathrm{FE}(\La):=\bigcup_{v\in\La^0}\left\{E\subseteq v\La\setminus \{v\}: E \mathrm{~is~finite ~ and ~ exhaustive} \right\}.$$
It is shown in \cite[Proposision 3.11]{far05} that if $E\in v\mathrm{FE}(\La)$ and $\mu\in v\La$, then $\mathrm{Ext}(\mu;E)$ is also finite and exhaustive.

\begin{defn}
Let $\La$ be a finitely aligned $k$-graph and let $R$ be a commutative unital ring. A collection $\{S_\la,S_{\la^*}: \la\in \La\}$ in an $R$-algebra $\mathcal{A}$ is called a \emph{Kumjian-Pask $\La$-family} if it satisfies the following relations:
\begin{enumerate}[(KP1)]
  \item $S_v S_w=\delta_{v,w} S_v$ for all $v,w\in\La^0$.
  \item $S_\mu S_\nu=S_{\mu \nu}$ and $S_{\nu^*}S_{\mu^*}=S_{(\mu \nu)^*}$ for all $\mu,\nu\in\La$ with $s(\mu)=r(\nu)$.
  \item $S_{\mu^*} S_\nu=\sum_{(\alpha,\beta)\in \La^{\mathrm{min}}(\mu,\nu)} S_\alpha S_{\beta^*}$ for all $\mu,\nu\in \La$.
  \item $\prod_{\mu\in E}\left(S_v-S_\mu S_{\mu^*} \right)=0$ for all $E\in v \mathrm{FE}(\La)$.
\end{enumerate}
\end{defn}

It is shown in \cite[Theorem 3.7]{cla16} that there is a (unique up to isomorphism) universal $R$-algebra $\kp$ generated by a Kumjian-Pask $\La$-family $\{s_\la, s_{\la^*}:\la\in\La\}$. This means that if $\{S_\la, S_{\la^*}:\la\in\La\}$ is a Kumjian-Pask $\La$-family in an $R$-algebra $\mathcal{B}$, then there exists a homomorphism $\pi:\kp \rightarrow \mathcal{B}$ such that $\pi(s_\la)=S_\la$ and $\pi(s_{\la^*})=S_{\la^*}$. We use lower-case letters for generating Kumjian-Pask families. It is a consequence of relations (KP1)-(KP4) that
$$\kp=\mathrm{span}\{s_\mu s_{\nu^*}:\mu,\nu\in\La \mathrm{~and ~} s(\mu)=s(\nu)\}.$$
Note that, by \cite[Theorem 3.7(c)]{cla16}, every Kumjina-Pask algebra $\kp$ is a $\Z^k$-graded ring with the grading components
$$\kp_n=\mathrm{span}\left\{s_\mu s_{\nu^*}: s(\mu)=s(\nu) \mathrm{~and~} d(\mu)-d(\nu)=n \right\} \hspace{.5cm} (n\in\Z^k)$$
(see also \cite[Theorem 3.4]{pin13} and \cite[Theorem 3.7(b)]{cla14}).

\subsection{Boundary paths}

Let $\La$ be a finitely aligned $k$-graph. For $n\in \N^k$, we denote by $\La^{\leq n}$ the set of all paths $\la$ with $d(\la)\leq n$ which cannot be extended to paths $\la\mu$ with $d(\la)<d(\la\mu)\leq n$; that is,
$$\La^{\leq n}:=\left\{\la\in\La: d(\la)\leq n, ~ \mathrm{and} ~ d(\la)_i<n_i ~\Longrightarrow ~ s(\la)\La^{e_i}=\emptyset \right\}.$$

To define the boundary paths in $\La$, we consider the following special $k$-graphs.

\begin{example}[\cite{kum00}]
Given $m\in(\N\cup \{\infty\})^k$, let $\Omega_{k,m}$ be the category
$$\Omega_{k,m}:=\left\{(p,q)\in \N^k\times\N^k : p\leq q\leq m \right\}$$
with $r(p,q):=(p,p)$ and $s(p,q):=(q,q)$. Then $\Omega_{k,m}$ equipped with the degree functor $d(p,q)=q-p$ is a $k$-graph. We usually denote $\Omega_{k,m}$ by $\Omega_k$ whenever $m=(\infty,\ldots,\infty)$.
\end{example}

Corresponding to each $\la\in\La$, we can define a degree-preserving functor $x_\la:\Omega_{k,d(\la)} \rightarrow \La$ such that $x_\la(p,q)=\la(p,q)$ for all $p\leq q\leq d(\la)$. Then the range of $x_\la$ is equal to the set $\{\la(p,q): p\leq q\leq d(\la)\}$ of subpaths of $\la$ in $\La$. Conversely, for every $m\in\N^k$ and graph morphism $x:\Omega_{k,m}\rightarrow \La$, we have $x=x_{x(0,m)}$. So, there is a one-to-on correspondence between the graph morphisms $x:\Omega_{k,m}\rightarrow \La$ and the elements of $\La^m$.

With this idea in mind, we recall the boundary paths of $\La$. Following \cite[Definition 5.10]{far05}, a \emph{boundary path} in $\La$ is a graph morphism $x:\Omega_{k,m} \rightarrow \La$ such that for all $(p,p)\in\Omega_{k,m}^0$ and all $E\in x(p,p)\mathrm{FE}(\La)$, there exists $\mu\in E$ with $x(p,p+d(\mu))=\mu$. We denote $\partial \La$ the set of all boundary paths in $\La$. The range map of $\La$ can be extended naturally to $\partial \La$ via $r(x):=x(0,0)$. For any $x\in \partial \La$, we also define $d(x):=m\in (\N\cup \{\infty\})^k$ \emph{the degree of $x$}. If $\La$ is both locally convex and row-finite, then $\partial \La$ coincides with the set
$$\La^{\leq \infty}:=\left\{x:\Omega_{k,m}\rightarrow \La: p\leq d(x) \mathrm{~and~} p_i=d(x)_i \mathrm{~impliy ~} x(p,p)\La^{e_i}=\emptyset \right\}$$
introduced in \cite{rae03}. However, we have $\La^{\leq\infty}\subseteq \partial \La$ with possibly nonequal in general. Recall also from \cite[Lemma 5.13]{far05} that $v\partial\La\neq \emptyset$ for all $v\in \La$.

For every $x\in \partial\La$ and $n\leq d(x)$, \emph{the shift of $x$} is the boundary path $\sigma^n(x)\in\partial \La$ such that $d(\sigma^n(x))=d(x)-n$ and $\sigma^n(x)(p,q):=x(n+p,n+q)$ for $p\leq q\leq d(x)-n$. Notice that the factorisation property implies $x(0,n)\sigma^n(x)=x$.

%%%%%%%%%%%%%%%%%%%%%%%%%%%%%%%%%%%%%%%%%%%%%%%

\section{Cofinality}\label{sec3}

Cofinality and aperiodicity are two key properties to characterize simple graph algebras and their generalizations. In this section, we focus on the cofinal $k$-graphs and give some descriptions in Theorem \ref{thm3.8} for ideal structure of associated Kumjian-Pask algebras.

\begin{defn}[{\cite[Definition 8.4]{sim06}}]
Let $\La$ be a finitely aligned $k$-graph. We say that $\La$ is \emph{cofinal} if for every $v\in \La^0$ and $x\in\partial \La$, there exists $n\leq d(x)$ such that $v\La x(n) \neq \emptyset$.
\end{defn}

If $\La$ is a locally convex row-finite $k$-graph, using \cite[Theorem 5.1]{pin13} and \cite[Theorem 9.4]{cla14}, basic graded ideals of $\kp$ may be completely characterized by saturated hereditary subsets of $\La^0$. However, in the non-row-finite case it seems that the structure of basic graded ideals of $\kp$ is more complicated (see \cite{tom07,lar12} for Leavitt path algebras and \cite{sim06} for higher-rank graph $C^*$-algebras). In Theorem \ref{thm3.8} below, we see that the confinality of $\La$ is equivalent to having only trivial saturated hereditary sets.

\begin{defn}
Let $\La$ be a finitely aligned $k$-graph.
\begin{enumerate}
  \item A subset $H\subseteq \La^0$ is called to be \emph{hereditary} if $v\in H$ and $v\La w\neq\emptyset$, we then have $w\in H$.
  \item A subset $H\subseteq \La^0$ is called to be \emph{saturated} if $E\subseteq v\mathrm{FE}(\La)$ with $s(E)\subseteq H$, we then have $v\in H$.
\end{enumerate}
\end{defn}

For convenience, we write $v\leq w$ whenever $v\La w\neq \emptyset$; that means, there exists a path $\la\in \La$ from $w$ to $v$. Recall from \cite[Lemma 3.2]{sim06} that for any hereditary set $H\subseteq \La^0$, the smallest saturated hereditary subset of $\La^0$ containing $H$ is
$$\overline{H}:=H\cup \left\{v\in\La^0\setminus H: \mathrm{~there ~ exists~} E\in v\mathrm{FE}(\La) ~\mathrm{with} ~ s(E)\subseteq H\right\}.$$
Moreover, if $H$ is hereditary and saturated, the restricted category $\La\setminus \La H:=(\La^0\setminus H,\La\setminus \La H,r,s,d)$ is a $k$-graph \cite[Lemma 4.1]{sim06}.

\begin{lem}\label{lem3.3}
Let $\La$ be a finitely aligned $k$-graph and let $H$ be a saturated hereditary subset of $\La^0$. If $E\in v\FE$ and $v\notin H$, then $E':=E\setminus EH\in v\mathrm{FE}(\La\setminus \La H)$.
\end{lem}

\begin{proof}
Notice that, since $v\notin H$, we have $E'\neq \emptyset$ by the saturation property. So, it suffices to show that $\mathrm{Ext}_{\La\setminus \La H}(\la; E')\neq \emptyset$ for all $\la \in v(\La\setminus \La H)$. To do this, fix arbitrary $\la\in v(\La\setminus \La H)$. Lemma 2.3 of \cite{sim06(2)} says that the set $\mathrm{Ext}_\La (\la;E)\subseteq s(\la)\La$ is finite and exhaustive. Since $s(\la)\notin H$, the saturation property of $H$ gives some $\la'\in \mathrm{Ext}_\La(\la;E)$ such that $s(\la')\notin H$. Thus we may factorise $\la \la'=\nu\beta$ with $\nu\in E$, $\beta\in \La\setminus \La H$ and $d(\la \la')=d(\la)\vee d(\nu)$. As $s(\nu)\notin H$ and $\nu\in E'=E\setminus EH$, we conclude that $\la'\in \mathrm{Ext}_{\La\setminus \La H}(\la;E')$, as desired.
\end{proof}

\begin{lem}\label{lem3.4}
Let $\La$ be a finitely aligned $k$-graph and let $H$ be a saturated hereditary subset of $\La^0$. Suppose that $\{T_\la, T_{\la^*}:\la\in \La\setminus \La H\}$ is a Kumjian-Pask $(\La \setminus \La H)$-family in an $R$-algebra $\mathcal{A}$. If we set
$$S_\la:=\left\{
           \begin{array}{ll}
             T_\la  & \la\in \La\setminus \La H \\
             0 & \la\in \La H
           \end{array}
         \right.
~~ \mathrm{and} ~~~~
S_{\la^*}:=\left\{
           \begin{array}{ll}
             T_{\la^*}  & \la\in \La\setminus \La H \\
             0 & \la\in \La H,
           \end{array}
         \right.
$$
then $\{S_\la, S_{\la^*}: \la\in\La\}$ is a Kumjian-Pask $\La$-family in $\mathcal{A}$.
\end{lem}

\begin{proof}
For the family $\{S_\la, S_{\la^*}: \la\in\La\}$, (KP1) is trivial, whereas (KP2) is a straightforward implication of the hereditary property of $H$.

For (KP3), fix $\mu,\nu\in \La$ and consider the following two cases:

{\bf Case 1:} Either $\mu$ or $\nu$ belongs to $\La H$. So, assume $\mu\in \La H$. Then $S_\mu=S_{\mu^*}=0$ and for every $(\alpha,\beta)\in \La^{\mathrm{min}}(\mu,\nu)$ we have $\mu\alpha\in \La H$ by the hereditariness. Thus, $s(\alpha)\in H$, $S_\alpha=0$, and $S_\alpha S_{\beta^*}=0$, and we get
$$\sum_{(\alpha,\beta)\in \Lamin(\mu,\nu)}S_\alpha S_{\beta^*}=0=S_{\mu^*}S_\nu.$$
The case $\nu\in \La H$ is similar.

{\bf Case 2:} Both $\mu,\nu$ belong to $\La\setminus \La H$. Since $\{T_\la, T_{\la^*}:\la\in \La\setminus \La H\}$ is a Kumjian-Pask $(\La\setminus \La H)$-family, we have
$$T_{\mu^*} T_\nu=\sum_{(\alpha,\beta)\in (\La\setminus \La H)^{\mathrm{min}}(\mu,\nu)}T_\alpha T_{\beta^*}$$
by (KP3). Moreover, for every $(\alpha,\beta)\in \Lamin(\mu,\nu)\setminus (\La\setminus \La H)^{\mathrm{min}}(\mu,\nu)$, we have $\mu\alpha\in \mathrm{MCE}_\La(\mu,\nu)\setminus \mathrm{MCE}_{\La\setminus \La H}(\mu,\nu)$ which follows $\mu\alpha\in \La H$ and $\alpha\in \La H$. So $S_\alpha=0$ by definition. Thus, we may compute
\begin{align*}
\sum_{(\alpha,\beta)\in \Lamin(\mu,\nu)}S_\alpha S_{\beta^*}&=\sum_{(\alpha,\beta)\in (\La\setminus \La H)^{\mathrm{min}}(\mu,\nu)}S_\alpha S_{\beta^*}\\
&=\sum_{(\alpha,\beta)\in (\La\setminus \La H)^{\mathrm{min}}(\mu,\nu)}T_\alpha T_{\beta^*}\\
&=T_{\mu^*}T_\nu=S_{\mu^*}S_\nu,
\end{align*}
and (KP3) holds for $S_{\mu^*}S_\nu$.

For (KP4), fix $E\in \FE$ and write $v:=r(E)$. If $v\in H$, then $\lambda\in \La H$ for every $\la\in E$, and (KP4) trivially holds for $E$. So, suppose $v\notin H$. Lemma \ref{lem3.3} implies that $E':=E\setminus EH$ is a finite and exhaustive set in the $k$-graph $\La\setminus \La H$. We may use (KP4) for $T_\la$'s to conclude that
\begin{align*}
\prod_{\la\in E}(S_v-S_\la S_{\la^*})&=\left(\prod_{\la\in E'}(S_v-S_\la S_{\la^*})\right)\left(\prod_{\la\in E H}(S_v-S_\la S_{\la^*})\right)\\
&=\left(\prod_{\la\in E'}(T_v-T_\la T_{\la^*})\right)(T_v)=0.
\end{align*}
This completes the proof.
\end{proof}

Given a subset $H\subseteq \La^0$, we write $I_H$ for the (two-sided) ideal of $\kp$ generated by $\{s_v:v\in H\}$. Furthermore, for an ideal $I$ of $\kp$ and $r\in R$, we define $H_{I,r}:=\{v\in \La^0: rs_v\in I\}$. If $r=1_R$, $H_{I,r}$ will be denoted by $H_I$. The proof of next lemma is similar to that of \cite[Lemma 3.3]{sim06}.

\begin{lem}\label{lem3.5}
Let $\La$ be a finitely aligned $k$-graph and let $R$ be a unital commutative ring. If $I$ is an ideal of $\kp$, then $H_{I,r}$ is a hereditary and saturated subset of $\La^0$ for every $r\in R$.
\end{lem}

The following is \cite[Lemma 5.4]{pin13} for finitely aligned $k$-graphs.

\begin{lem}\label{lem3.6}
Let $H$ be a hereditary subset of $\La^0$. Then
\begin{equation}\label{equ3.1}
I_H=\mathrm{span}\{s_\mu s_{\nu^*}:\mu,\nu\in\La, ~ s(\mu)=s(\nu)\in H\},
\end{equation}
that is a graded ideal of $\kp$ and we have $H_{I_H}=\overline{H}$.
\end{lem}

\begin{proof}
Denote the right-hand side of equation \ref{equ3.1} by $J$.  Since $H\subseteq H_{I_H}$, we have $s_\mu s_{\nu^*}=s_\mu s_{s(\mu)}s_{\nu^*}\in I_H$ for every $\mu,\nu\in \La$ with $s(\mu)=s(\nu)\in H$, and hence $J\subseteq I_H$.

For the reverse, since $J$ contains the generators $\{s_v:v\in H\}$, it suffices to show that $J$ is an ideal of $\kp$. For this, fix some $s_\mu s_{\nu^*}$ with $s(\mu)=s(\nu)\in {H}$ and some $s_\la s_{\sigma^*}\in \kp$. Using (KP3), we get
\begin{equation}\label{equ3.2}
s_\mu s_{\nu^*}s_\la s_{\sigma^*}=s_\mu (s_{\nu^*}s_\la) s_{\sigma^*}=\sum_{(\rho,\tau)\in\Lamin(\nu,\la)} s_{\mu\rho}s_{(\sigma\tau)^*}.
\end{equation}
Since $H$ is hereditary, the facts $r(\rho)=s(\mu)$ and $r(\tau)=s(\sigma)$ yield that $s(\rho),s(\tau)\in H$. Thus each nonzero summand in (\ref{equ3.2}) lies in $J$, so does $s_\mu s_{\nu^*} s_\la s_{\sigma^*}$. By a same argument, we have also $s_\la s_{\sigma^*}s_\mu s_{\nu^*}\in J$. Therefore, $J$ is an ideal of $\kp$ and equation \ref{equ3.1} follows. Note that $I_H$ is a graded ideal as generated by homogenous elements.

Now we show $H_{I_H}=\overline{H}$. Since $H\subseteq H_{I_H}$, Lemma \ref{lem3.5} gives $\overline{H}\subseteq H_{I_H}$. To see $H_{I_H}\subseteq\overline{H}$, we show that $v\notin \overline{H}$ implies $s_v\notin I_H$ for every $v\in \La^0$. To do this, consider the $k$-graph $\La\setminus \La \overline{H}$ and let $\mathrm{KP}(\La\setminus \La \overline{H})$ generated by a Kumjian-Pask $\La$-family $\{t_\la,t_{\la^*}:\la\in\La \overline{H}\}$. Using Lemma \ref{lem3.4}, we may extended this family to a Kumjian-Pask $\La$-family $\{T_\la,T_{\la^*}:\la\in\La\}$. Then, by the universality, there is a homomorphism
$$\left\{
    \begin{array}{l}
      \phi:\kp\longrightarrow\mathrm{KP}(\La\setminus \La \overline{H}) \\
      \hspace{.85cm}r s_\mu s_{\nu^*}\longmapsto rT_\mu T_{\nu^*}  \hspace{1cm} (\mu,\nu\in\La ~ \mathrm{and} ~ s(\mu)=s(\nu)).
    \end{array}
  \right.
$$
Since $\phi(s_v)=0$ for $v\in H$, $\phi$ vanishes on $I_H$. On the other hand, for each $v\in \La^0\setminus \overline{H}$, we have $\phi(s_v)=t_v\neq 0$ and hence $s_v\notin \ker \phi\supseteq I_H$. Now the result follows.
\end{proof}

We will also need the next lemma to prove Theorem \ref{thm3.8}. To prove it, we use some terminology of \cite[Section 3]{rae04}. Fix a finite set $E\subseteq \La$. By \cite[Lemma 3.2]{rae04}, there exists a finite set $F\subseteq \La$ containing $E$ which satisfies
\begin{equation}\label{equ3.3}
\la,\mu,\rho,\tau\in F,~ d(\la)=d(\mu), ~ d(\rho)=d(\tau), ~s(\la)=s(\mu), ~\mathrm{and}~ s(\rho)=s(\tau)
\end{equation}
\[\mathrm{imply} ~ \left\{\la\alpha,\tau\beta:(\alpha,\beta)\in\La^{\mathrm{min}}(\mu,\rho)\right\}.\]
We define the finite sets
\[\Pi E:=\bigcap \left\{F\subseteq \La: E\subseteq F ~ \mathrm{and} ~ F ~ \mathrm{satisfies ~ (\ref{equ3.3})} \right\}\]
and
\[\Pi E\times_{d,s} \Pi E:=\left\{(\la,\mu)\in\Pi E\times \Pi E: d(\la)=d(\mu), s(\la)=s(\mu)\right\}.\]
Then similar to \cite[Lemma 3.2]{rae04}, we may show that
\[M_{\Pi E}^s:=\mathrm{span}_R\left\{ s_\la s_{\mu^*}: (\la,\mu)\in\Pi E\times_{d,s}\Pi E\right\}\]
is a finite-dimensional subalgebra of $\kp_0$. Recall also from \cite[Lemma 4.2]{cla16} that $M_{\Pi E}^s$ is spanned by $\{\Theta(s)_{\la,\mu}^{\Pi E}: (\la,\mu)\in \Pi E\times_{d,s}\Pi E\}$, where
\[\Theta(s)_{\la,\mu}^{\Pi E}:=s_\la \left(\prod_{\substack{\la\nu\in \Pi E\\ d(\nu)\neq 0}}\left(s_{s(\la)}-s_{\la\nu}s_{(\la\nu)^*}\right)\right)s_{\mu^*}.\]

\begin{lem}[{See \cite[Proposition 6.3]{pin13}}]\label{lem3.7}
Let $\La$ be a finitely aligned $k$-graph and $R$ a unital commutative ring. If $J$ is an ideal of $R$, then we have the following:
\begin{enumerate}
  \item If $rs_v\in J\kp$, then $r\in J$.
  \item If $I$ is an ideal of $\kp$ such that $J\kp\subseteq I$, then the action
$$(r+J)(x+I):=rx+I  \hspace{1cm} (\mathrm{for}~ r\in R,~ x\in \kp)$$
forms the quotient $\kp/I$ as an $R/J$-algebra.
  \item We have
$$\frac{\kp}{J\kp}\cong \mathrm{KP}_{\frac RJ}(\La)$$
 as $R/J$-algebras.
\end{enumerate}
\end{lem}

\begin{proof}
For (1), suppose that $rs_v\in J\kp$ for some $r\in R$ and $v\in \La^0$. If $rs_v=0$, we must have $r=0\in J$ by \cite[Theorem 3.7(b)]{cla16}. So assume $rs_v\neq 0$. Since $rs_v\in (J\kp)_0=J\kp_0$, there is a finite span
$$rs_v=\sum_{(\alpha,\beta)\in F}r_{\alpha,\beta}s_\alpha s_{\beta^*}$$
for $rs_v$, where $r_{\alpha,\beta}\in J$ and $d(\alpha)=d(\beta)$ for every $(\alpha,\beta)\in F$. Thus we may consider the matricial subalgebra $M_{\Pi F}^s$ of $\kp_0$ and see that $rs_v\in J M_{\Pi F}^s$. Recall from \cite[Lemma 4.2]{cla16} that $\{\Theta(s)^{\Pi F}_{\la,\mu}: (\la, \mu)\in \Pi F\times_{d,s}\Pi F\}$ forms a set of matrix units which spans $M_{\Pi F}^s$. Hence one can write
$$rs_v=\sum_{(\la, \mu)\in \Pi F\times_{d,s}\Pi F} r_{\la,\mu} \Theta(s)^{\Pi F}_{\la,\mu}  \hspace{1cm} (r_{\la,\mu}\in J).$$
In particular, we have $r(\la)=r(\mu)=v$ for every $r_{\la,\mu}\neq 0$. Fix some $(\sigma,\gamma)\in \Pi F\times_{d,s} \Pi F$ with $r(\la)=r(\mu)=v$; since $rs_v\neq 0$, such $(\sigma,\gamma)$ exists. We then compute
\begin{align*}
r \Theta(s)_{\gamma,\gamma}^{\Pi F}&=\left(\Theta(s)_{\sigma,\gamma}^{\Pi F}\right)^* rs_v \left(\Theta(s)_{\sigma,\gamma}^{\Pi F}\right)\\
&=\left(\Theta(s)_{\sigma,\gamma}^{\Pi F}\right)^*\left(\sum_{(\la,\mu)}r_{\la,\mu} \Theta(s)_{\la,\mu}^{\Pi F}\right)\Theta(s)_{\sigma,\gamma}^{\Pi F}\\
&=\sum_{(\la,\mu)}r_{\la,\mu} \left(\Theta(s)_{\sigma,\gamma}^{\Pi F}\right)^*\left(\Theta(s)_{\la,\mu}^{\Pi F}\Theta(s)_{\sigma,\gamma}^{\Pi F}\right)\\
&=\sum_{(\la,\mu)}r_{\la,\mu} \left(\Theta(s)_{\sigma,\gamma}^{\Pi F}\right)^*\left(\delta_{\mu,\sigma}\Theta(s)_{\la,\gamma}^{\Pi F}\right)\\
&=\sum_{(\la,\mu)}r_{\la,\mu} \delta_{\sigma,\la}\delta_{\mu,\sigma}\Theta(s)_{\gamma,\gamma}^{\Pi F}\\
&=r_{\sigma,\sigma} \Theta(s)_{\gamma,\gamma}^{\Pi F}.
\end{align*}
Therefore, $(r-r_{\sigma,\sigma})\Theta(s)_{\gamma,\gamma}^{\Pi F}=0$ which follows $r=r_{\sigma,\sigma}\in J$ because $\Theta(s)_{\gamma,\gamma}^{\Pi F}\neq 0$. This proves statement (1).

For (2), it suffices to show that the action of $R/J$ on $\kp/I$ is well-defined. Indeed, if $r+J=s+J$ and $x+I=y+I$, then
$$rx-sy=r(x-y)+(r-s)y\in I+J\kp \subseteq I.$$
So, we get $rx+I=sy+I$, as desired.

Finally we prove statement (3). Let $\{t_\la,t_{\la^*}:\la\in\La\}$ be a Kumjian-Pask $\La$-family generating $\mathrm{KP}_{R/J}(\La)$ and view the quotient $\kp/J\kp$ as an $R/J$-algebra by part (2). Since $\{s_\la+J\kp, s_{\la^*}+J\kp:\la\in\La\}$ is a Kumjian-Pask $\La$-family, the universality gives a homomorphism $\phi:\mathrm{KP}_{R/J}(\La)\rightarrow \kp/J\kp$ such that $\phi((r+J)t_\la)=rs_\la+J\kp$ and $\phi((r+J)t_{\la^*})=rs_{\la^*}+J\kp$. Note that $J\kp$ is a graded ideal of $\kp$, so $\kp/J\kp$ is $\Z^k$-graded with the grading components
$$\left(\frac{\kp}{J\kp}\right)_n:=\frac{\kp_n}{J\kp}  \hspace{1cm} (n\in \Z^k).$$
Thus $\phi$ is a graded homomorphism because each grading component $\mathrm{KP}_{R/J}(\La)_n$ is embedded into $(\kp/J\kp)_n$. Moreover, part (1) implies that $\phi((r+J)t_v)\neq 0$ for every $v\in \La^0$ and $r\in R\setminus J$. Now apply the graded uniqueness theorem \cite[Theorem 4.1]{cla16} to obtain the injectivity of $\phi$. As $\phi$ is surjective either, we conclude that it is an $R/J$-algebra isomorphism from $\mathrm{KP}_{R/J}(\La)$ onto $\kp/J\kp$.
\end{proof}

\begin{thm}\label{thm3.8}
Let $\La$ be a finitely aligned $k$-graph and $R$ a unital commutative ring. Then the following statements are equivalent.
\begin{enumerate}
  \item $\La$ is cofinal.
  \item The only saturated hereditary subsets of $\La^0$ are $\emptyset$ and $\La^0$.
  \item Every graded ideal of $\kp$ is of the form $J\kp$ for some ideal $J$ of $R$.
  \item For every ideal $I$ of $\kp$, there exists an ideal $J$ of $R$ such that $I\cap \kp_0\subseteq J\kp\subseteq I$.
  \item For every ideal $I$ of $\kp$containing some $s_v$, where $v\in \La^0$, we have $I=\kp$.
\end{enumerate}
\end{thm}

\begin{proof}
(1) $\Longrightarrow$ (2): Let $\La$ be cofinal and suppose on the contrary that $H\neq \emptyset, \La^0$ is a nontrivial saturated hereditary subset of $\La^0$. The saturation and hereditary properties of $H$ are equivalent to the conditions (i) and (ii) of \cite[Lemma 5.2]{lew10} for $K:=\La^0\setminus H$, respectively. Since $\La^0\setminus H\neq \emptyset$, \cite[Lemma 5.2]{lew10} gives some $x\in\partial \La$ such that $x(n)\in \La^0\setminus H$ for every $n\leq d(x)$. Take some $v\in H$. By the cofinality, there exists $E\in x(0)\mathrm{FE}(\La)$ such that $v\La s(\alpha)\neq \emptyset$ for every $\alpha\in E$. Since $x\in \partial \La$, there is $\alpha\in E$ such that $x(0,d(\alpha))=\alpha$. In particular, $x(d(\alpha))=s(\alpha)$ and $v\La x(d(\alpha))\neq \emptyset$. On the other hand, we have $v\in H$ and so $x(d(\alpha))\in H$ by the hereditary property. This contradicts the choice of $x$.

(2) $\Longrightarrow$ (1): Suppose that $\emptyset$ and $\La^0$ are the only saturated hereditary subsets of $\La^0$. Take some $v\in H$ and $x\in\partial \La$. If we define $H_v:=\{s(\la):\la\in v\La\}$, then $H_v$ is hereditary and by \cite[Lemma 3.2]{sim06} its saturation is
$$\overline{H_v}=H\cup \left\{ w\in \La^0: \mathrm{there ~ exists ~} E\in w\mathrm{FE}(\La) ~\mathrm{such ~ that~} s(E)\subseteq H_v\right\}.$$
Apply statement (2) to get $\overline{H_v}=\La^0$. So $r(x)\in \overline{H_v}$ and there exists a finite exhaustive set $E$ with $s(E)\subseteq H_v$. Note that if $r(x)\in H$, then $\{r(x)\}$ is finite and exhaustive. Since $x$ is a boundary path, there is $\mu\in E$ such that $x(0,d(\mu))=\mu$. In particular, we have $\mu\in r(x)\La s(\la)$ for some $\la\in v\La$. Hence, $\la\in v\La s(\la)=v \La s(\mu)$ and so $v \La x(d(\mu))\neq \emptyset$. As $v$ and $x$ were arbitrary, we conclude the cofinality of $\La$.

(2) $\Longrightarrow$ (3): Let $I$ be a graded ideal of $\kp$. If we set
$$J:=\{r\in R: rs_v\in I ~ \mathrm{for ~ some ~ } v\in \La^0\},$$
then $J$ is an ideal of $R$. We show that $I=J\kp$. Note that for every $r\in R$, the set $H_{I,r}=\{v\in \La^0:rs_v\in I\}$ is saturated and hereditary. If $r\in J$, then $H_{I,r}\neq \emptyset$ and hence we have $H_{I,r}=\La^0$ and $r\kp \subseteq I$. This follows $J\kp \subseteq I$.

For the reverse containment, we consider the quotient map
$$\left\{
    \begin{array}{l}
      \phi:\frac{\kp}{J\kp} \rightarrow \frac{\kp}{I} \\
      x+J \kp \mapsto x +I  \hspace{1cm} \mathrm{~ for ~}x\in \kp
    \end{array}
  \right.
$$
which is well-defined because $J\kp \subseteq I$. Let $\pi:\mathrm{KP}_{R/J}(\La) \rightarrow \kp/J\kp$ be the isomorphism of Lemma \ref{lem3.7}(3). Note that if $\{t_\la,t_{\la^*}:\la\in \La\}$ is a generating family for $\mathrm{KP}_{R/J}(\La)$, we have $\phi\circ \pi((r+J)t_v)=rs_v+I\neq I$ for all $r\in R\setminus J$ and $v\in \La^0$. Also, since both $I$ and $J\kp$ are graded ideals of $\kp$, then $I/J\kp$ is a graded ideal of $\kp/J\kp\cong \mathrm{KP}_{R/J}(\La)$. Hence the quotient
$$\frac\kp I\cong \frac{\kp/J\kp}{I/J\kp}$$
is a graded $R/J$-algebra and $\phi\circ \pi$ is a graded homomorphism. Now we may apply the gauge invariant uniqueness theorem \cite[Theorem 4.1]{cla16} for $\phi\circ\pi:\mathrm{KP}_{R/J}(\La)\rightarrow \kp/I$ to conclude that $\phi\circ\pi$ is injective. Therefore, $\phi$ is injective, and we get $I=J\kp$ as desired.

(3) $\Longrightarrow$ (4): Suppose that $I$ is an ideal of $\kp$ and $I_0$ is the ideal of $\kp$ generated by $I\cap \kp_0$. Then $I_0$ is a graded ideal because generated by homogeneous elements. By statement (3), there exists an ideal $J$ of $R$ such that $I_0=J\kp$, and so, we have $I\cap \kp_0\subseteq J\kp=I_0 \subseteq I$.

(4) $\Longrightarrow$ (5): If $I$ is an ideal of $\kp$ containing some $s_v$, then $H_I:=\{w\in \La^0:s_w\in I\}$ is a nonempty saturated hereditary subset of $\La^0$ and we have $I_{H_I}\subseteq I$. Since $I_{H_I}$ is a graded ideal by Lemma \ref{lem3.5}, statement (4) implies that there is an ideal $J$ of $R$ so that $J\kp= I_{H_I}$. But $1_R.s_v=s_v\in I_{H_I}=J\kp$ which follows $1_R\in J$ by Lemma \ref{lem3.7}(1). Hence $J=R$ and we get $I=I_{H_I}=\kp$.

(5) $\Longrightarrow$ (2): Suppose that $H$ is a nonempty saturated hereditary subset of $\La^0$. If $v\in H$, then $I_H$ is an ideal of $\kp$ containing $s_v$. So we have $I_H=\kp$ by statement (5) which follows $H_{I_H}=\La^0$. On the other hand, Lemma \ref{lem3.5} says that $H_{I_H}=H$ and hence $H=\La^0$. It follows (2) and completes the proof.
\end{proof}

%%%%%%%%%%%%%%%%%%%%%%%%%%%%%%%%%%%%%%%%%%%%%%%%

\section{aperiodicity and generalized cycles}\label{sec4}

There are several aperiodicity conditions in the literature which are equivalent (see \cite[Proposition 2.11]{sho12} and \cite[Proposition 3.6]{lew10}). Here, we consider the following from \cite{rob07,sho12}.

\begin{defn}
Let $\La$ be a finitely aligned $k$-graph. For $v\in\La^0$ and $m\neq n\in\N^k$, we say $\La$ has a \emph{local periodicity $m,n$ at $v$} if for every $x\in v\partial \La$ we have $m\vee n\leq d(x)$ and $\sigma^m(x)=\sigma^n(x)$. We say that $\La$ is \emph{aperiodic} if $\La$ has no local periodicity at all $v\in\La^0$; that is, for every $v\in \La^0$ and $m\neq n\in\N^k$, there exists $x\in v\partial \La$ such that either $m\vee n\nleq d(x)$ or $\sigma^m(x)\neq \sigma^n(x)$.
\end{defn}

In \cite{eva02}, Evans introduces the notion of generalized cycle for higher-rank graphs. In spite of ordinary 1-graphs, we could not completely describe the aperiodicity of $k$-graphs by properties of generalized cycles. However, we give some relations between generalized cycles and the aperiodicity in Corollary \ref{cor4.5} below.

\begin{defn}[{See \cite{eva02,eva12}}]
Let $\La$ be a finitely aligned $k$-graph. A \emph{generalized cycle} in $\La$ is a pair $(\mu,\nu)$ of distinct paths in $\La$ such that $s(\mu)=s(\nu)$, $r(\mu)=r(\nu)$ and $\mathrm{MCE}(\mu\tau, \nu)\neq \emptyset$ for all $\tau\in s(\mu)\La$. A path $\tau \in s(\nu)\La$ is called an \emph{entrance} for $(\mu,\nu)$ whenever $\mathrm{MCE}(\mu,\nu\tau)=\emptyset$.
\end{defn}

We use \cite[Lemma 3.2]{eva12} to prove Lemmas \ref{lem4.3} and \ref{lem4.4} below, which says that a pair $(\mu,\nu)$ with $s(\mu)=s(\nu)$ and $r(\mu)=r(\nu)$ is a generalized cycle if and only if the set $\mathrm{Ext}\left(\mu;\{\nu\}\right)$ is finite and exhaustive.

The proof of next lemma is analogous with that of \cite[Lemma 3.7]{eva12} with a small modification.

\begin{lem}\label{lem4.3}
Let $\La$ be a finitely aligned $k$-graph. If $(\mu,\nu)$ is a generalized cycle in $\La$, then $s_\mu s_{\mu^*}\leq s_\nu s_{\nu^*}$ (in the sense that $(s_\mu s_{\mu^*})(s_\nu s_{\nu^*})=(s_\nu s_{\nu^*})(s_\mu s_{\mu^*})=s_\mu s_{\mu^*}$). Furthermore, $(\mu,\nu)$ has no entrances if and only if $s_\mu s_{\mu^*}=s_\nu s_{\nu^*}$.
\end{lem}

\begin{lem}\label{lem4.4}
Let $\La$ have a local periodicity $m$, $n$ at $v$. Then for every $\mu\in v\La^{m\vee n}$, there is a unique $\nu\in v\La^{n-m+m\vee n}$ such that $(\mu,\nu)$ is a generalized cycle. Moreover, such generalized cycles have no entrances.
\end{lem}

\begin{proof}
Fix $\mu\in v\La^{m\vee n}$ and let $x\in s(\mu)\partial \La$. Then $\mu x\in v\partial \La$ by \cite[Lemma 5.13]{far05}. Note that the local periodicity implies $\sigma^m(\mu x)=\sigma^n(\mu x)$, and so, $d(\mu x)-m=d(\mu x)-n$. In particular, we have $d(\mu x)_i=\infty$ for each index $i$ with $m_i\neq n_i$. We set $\nu:=\mu x(0,n-m+m\vee n)$ and show that $(\mu,\nu)$ is a generalized cycle without entrances. Again, the local periodicity gives that
\begin{equation}\label{equ4.1}
\nu=\mu x(0,n)\mu x(n,n-m+m\vee n)=\mu x(0,n)\mu(m,m\vee n).
\end{equation}
Hence, $s(\nu)=s(\mu)$ and $r(\nu)=r(\mu)$. To see $(\mu,\nu)$ is a generalized cycle without without entrances, we show that $\mathrm{MCE}(\mu\tau,\nu)\neq \emptyset$ and $\mathrm{MCE}(\mu,\nu\tau)\neq \emptyset$ for all $\tau\in s(\mu)\La$. For this, take arbitrary $\tau\in s(\mu)\La$. Let $y\in s(\tau)\partial\La$ and define $z:=\tau y\in \partial \La$. Then
\begin{align*}
\mu z&=\mu z(0,n)\sigma^n(\mu z)\\
&=\mu z(0,n)\sigma^m(\mu z)=\mu z(0,n) \mu z(m,m\vee n)\sigma^{m\vee n}(\mu z)\\
&=\mu(0,n)\mu(m,m\vee n)z=\nu z \hspace{1cm} (\mathrm{by~ \ref{equ4.1}}).
\end{align*}
We also have
$$\mu z(0,m\vee n+d(\tau))=\mu \tau y(0,m\vee n+d(\tau))=\mu \tau,$$
$$\mu z(0,n-m+m\vee n)=\nu z(0,n-m+m\vee n)=\nu$$
and hence, $\mu z(0,d(\mu\tau)\vee d(\nu))\in \mathrm{MCE}(\mu\tau,\nu)$.

Similarly, since $\mu z=\nu z=\nu\tau y$, we have
$$\mu z(0,(m\vee n)\vee(d(\tau)+n-m+m\vee n))=\mu z(0,(n-m+d(\tau))+m\vee n)\in\mathrm{MCE}(\nu\tau,\mu).$$
Therefore, both $\mathrm{MCE}(\mu\tau,\nu)$ and $\mathrm{MCE}(\nu\tau,\mu)$ are nonempty which imply that $(\mu,\nu)$ is a generalized cycle without any entrance.

To complete the proof, we show that such morphism $\nu$ is unique in $v\La^{n-m+m\vee n}$. Indeed, if $(\mu,\nu)$ and $(\mu,\nu')$ are two generalized cycles such that $\nu,\nu'\in v\La^{n-m+m\vee n}$, then $s_\nu s_{\nu^*}=s_\mu s_{\mu^*}=s_{\nu'}s_{\nu'^*}$ by Lemma \ref{lem4.3}. This turns out $0\neq s_{\nu^*}s_{\nu'}=\delta_{\nu,\nu'}s_{s(\nu)}$, and so $\nu=\nu'$ as desired.
\end{proof}

As usual, we say $\mu\in\La^{\neq 0}$ is a \emph{cycle} whenever $s(\mu)=r(\mu)$. Note that if $\mu$ is a cycle, then $(\mu,\{s(\mu)\})$ would be a generalized cycle in $\La$. Following \cite{eva12}, a cycle $\mu\in\La$ is called an \emph{initial cycle} if we have $r(\mu)\La^{e_i}=\emptyset$ whenever $d(\mu)_i=0$.

\begin{cor}\label{cor4.5}
Let $\La$ be a finitely aligned $k$-graph. Then
\begin{enumerate}
  \item If every generalized cycle has an entrance, $\La$ is aperiodic.
  \item If $\La$ is aperiodic, every initial cycle in $\La$ has an entrance.
\end{enumerate}
\end{cor}

\begin{proof}
Statement (1) follows from Lemma \ref{lem4.4}. For (2), we prove the contrapositive statement. So, assume $\mu$ is an initial cycle with no entrances. Then by the factorisation property, there exists a unique functor $\mu^\infty$ such that $\mu^\infty(ld(\mu),(l+1)d(\mu))=\mu$ for all $l\in\mathbb{N}$. Since $r(\mu)\La^{e_i}=\emptyset$ whenever $d(\mu)_i=0$, similar to proof of \cite[Lemma 2.11]{rae03} we may show that the graph morphism $\mu^\infty$ is a boundary path. Take an arbitrary $x\in r(\mu)\partial \La$. As $\mu$ has no entrances, we must have $d(\mu)\leq d(x)$ and $\mathrm{MCE}(\mu,x(0,d(\mu)))\neq\emptyset$, which follows $x(0,d(\mu))=\mu$ and $x=\mu\sigma^{d(\mu)}(x)$. An inductive argument shows also that $x=\mu^l \sigma^{ld(\mu)}(x)$ for every $l\in \mathbb{N}$, and hence $x=\mu^\infty$. Since $x$ was arbitrary, we conclude that $r(\mu)\partial\La=\{\mu^\infty\}$, and therefore, $\La$ has a local periodicity $0$, $d(\mu)$ at $r(\mu)$.
\end{proof}

We know known that $\La$ is aperiodic and cofinal if and only if $\kp$ is basically simple (see \cite[Theorem 5.14]{pin13}, \cite[Theorem 8.5]{cla14}, and \cite[Theorem 9.3]{cla16}). In the following, we describe the ideal structure of basically simple Kumjian-Pask algebras.

\begin{prop}\label{prop4.6}
Let $\La$ be a finitely aligned $k$-graph and $R$ a unital commutative ring. Then the following are equivalent:
\begin{enumerate}
  \item $\La$ is aperiodic and cofinal.
  \item $\La$ is aperiodic and the only saturated hereditary subsets of $\La^0$ are $\emptyset$ and $\La^0$.
  \item Every ideal of $\kp$ is of the form $J\kp$ for some ideal $J$ of $R$.
\end{enumerate}
\end{prop}

\begin{proof}
The implication (1) $\Rightarrow$ (2) follows from Theorem \ref{thm3.8}. For (2) $\Rightarrow$ (3), let $I$ be an ideal of $\kp$. If we define
$$J:=\{r\in R: rs_v\in I ~ \mathrm{for ~ some ~} v\in \La^0 \},$$
then $J$ is an ideal of $R$. Similar to the proof of (2) $\Rightarrow$ (3) in Theorem \ref{thm3.8}, we can show $J\kp\subseteq I$. Notice that we have $rs_v\notin I$ for every $r\in R\setminus J$ and $v\in \La^0$. Consider the quotient homomorphism $q:\kp/J\kp \rightarrow \kp/I$ and let us denote $\pi:\mathrm{KP}_{R/J}(\La) \rightarrow \kp/J\kp$ the isomorphism of Lemma \ref{lem3.7}(3). If $\{t_\la, t_{\la^*}:\la\in\La\}$ is a generating $\La$-family for $\mathrm{KP}_{R/J}(\La)$, we have $q\circ\pi((r+J)t_v)=q(rs_v+J\kp)\neq I$ for all $r\in R\setminus J$ and $v\in\La^0$. Then the Cuntz-Krieger uniqueness theorem \cite[Theorem 8.1]{cla16} implies that $q\circ \pi$ is injective. Thus, $q$ is injective either, and we get $I=J\kp$, as desired.

(3) $\Longrightarrow$ (1). Suppose that the statement (3) holds. By Theorem \ref{thm3.8}(3), $\La$ is cofinal. So, we show that $\La$ is aperiodic. Let $\pi_s:\kp \rightarrow \mathrm{End}(\mathbb{F}_R(\partial \La))$ be the boundary path representation of \cite[Definition 3.9]{cla16}. \cite[Proposition 3.6]{cla16} implies that $\pi_s(rs_v)\neq 0$ for every $r\in R\setminus \{0\}$ and $v\in \La^0$, and so, the ideal $\ker\pi_s$ contains no elements $rs_v$ for $r\in R\setminus \{0\}$ and $v\in \La^0$. By statement (3), we must have $\ker \pi_s=(0)$, and hence $\pi_s$ is injective. Now we apply \cite[Corolary 8.3]{cla16} to obtain the aperiodicity of $\La$.
\end{proof}

%%%%%%%%%%%%%%%%%%%%%%%%%%%%%%%%%%%%%%%%%%%%%%%%

\section{Purely infinite simple Kumjian-Pask algebras}\label{sec6}

In this section, we use the notion of generalized cycles to give some conditions for $R$ and $\La$ in Theorem \ref{thm6.4} under which the associated Kumjian-Pask algebra $\kp$ is purely infinite simple. Since every ordinary cycle in $\La$ may be considered as a generalized cycle, this result is the extension of \cite[Theorem 11]{abr06} and \cite[Theorem 3.1]{lar13} to Kumjian-Pask algebras.

Let $A$ be a ring. Recall from \cite[Definition 1.2]{ara02} that an idempotent $e$ in $A$ is called \emph{infinite} if $eA$ is isomorphic to a proper direct summand of itself as right $A$-modules. We say that $A$ is \emph{purely infinite} if every nonzero right ideal of $A$ contains an infinite idempotent. It is a consequence of \cite[Theorem 1.6]{ara02} that the definition of purely infinite simple rings is left-right symmetric.

We need the next lemma, Lemma \ref{lem5.2}, to prove Theorem \ref{thm6.4} which is an extension of \cite[Proposition 4.9]{pin13} to finitely aligned $k$-graphs. However, it seems that we cannot apply the computations of \cite{pin13} in the non-row-finite setting. We use the theory of Steinberg algebras to prove Lemma \ref{lem5.2}.

Let us first briefly review from \cite[Section 5]{cla16} the construction of Steinberg algebra $A_R(\mathcal{G}_\La)$ associated to a Kumjian-Pask algebra $\kp$. Fix a finitely aligned $k$-graph $\La$. Associated to the boundary path space $\partial \La$, we define the groupoid $\mathcal{G}_\La$ such that
\begin{align*}
\mathrm{Obj}(\mathcal{G}_\La)&:=\partial \La\\
\mathrm{Mor}(\mathcal{G}_\La)&:=\left\{(\la z,d(\la)-d(\mu), \mu z): \la,\mu\in\La, s(\la)=s(\mu), z\in s(\la)\partial\La \right\},
\end{align*}
and the range and source maps are defined by $r(x,m,y):=x$ and $s(x,m,y):=y$. Moreover, the composition and inversion are as follows
\begin{align*}
(x,m,y)\circ (y,n,z)&:=(x,m+n,z), \mathrm{~and}\\
(x,m,y)^{-1}&:=(y,-m,x).
\end{align*}
For every $\la\in\La$ and non-exhaustive set $G\subseteq s(\la)\La$, write $Z(\la):=\la\partial \La$ and
$$Z(\la\setminus G):=Z(\la)\setminus \left(\bigcup_{\nu\in G} Z(\la\nu)\right).$$
If $s(\mu)=s(\la)$ we define
$$Z(\la*_s \mu\setminus G):=Z(\la*_s\mu)\setminus \left(\bigcup_{\nu\in G}Z(\la\nu*_s\mu\nu)\right),$$
where
$$Z(\la*_s \mu):=\left\{(x,d(\la)-d(\mu),y):x\in Z(\la),y\in Z(\mu) \mathrm{~and~} \sigma^{d(\la)}(x)=\sigma^{d(\mu)}(y)\right\}.$$
Then the sets $Z(\la*_s\mu\setminus G)$ form a basis of compact open elements for a second-countable, Hausdorff topology on $\mathcal{G}_\La$.

\begin{defn}
Let $\La$ be a finitely aligned $k$-graph and $R$ a unital commutative ring. The $R$-algebra
$$A_R(\mathcal{G}_\La):=\left\{f:\mathcal{G}_\La\rightarrow R: f ~\mathrm{is~locally~constant~with~compact~support}\right\}$$
with pointwise addition, scalar multiplication, and the convolution
$$f*g(t):=\sum_{\substack{s,t\in \mathcal{G}_\La\\ r(t)=r(s)}}f(t)g(s^{-1}t)$$
is called the \emph{Steinberg algebra associated to $\La$}.
\end{defn}

By \cite[Proposition 5.4]{cla16}, there is an $R$-algebra isomorphism $\pi:\kp\rightarrow A_R(\mathcal{G}_\La)$ such that $\pi(s_\la)=1_{Z(\la*_s s(\la))}$ and $\pi(s_{\la^*})=1_{Z(s(\la)*_s\la)}$ for all $\la\in\La$.

\begin{lem}\label{lem5.2}
Let $\La$ be an aperiodic finitely aligned $k$-graph. Then for every nonzero element $a\in \kp$, there exist $c,d\in \kp$ such that $cad=rs_v$ for some $r\in R\setminus\{0\}$ and $v\in \La^0$.
\end{lem}

\begin{proof}
We use the argument of \cite[Theorem 3.2]{cla15}. Suppose $\pi:\kp\rightarrow A_R(\mathcal{G}_\La)$ is the isomorphism of \cite[Proposition 5.4]{cla16}. Then $\pi(a)$ is a nonzero element of $A_R(\mathcal{G}_\La)$ and we may apply \cite[Lemma 3.1]{cla15} to get a compact open set $B$ such that $f:=1_{B^{-1}}\ast \pi(a)$ is nonzero on $\mathcal{G}_\La^{(0)}$. Since $\mathcal{G}_\La^{(0)}$ is both open and closed, the function
$$f_0(t):=\left\{
  \begin{array}{ll}
    f(t) & t\in\mathcal{G}_\La^{(0)} \\
    0 & t\in\mathcal{G}_\La\setminus \mathcal{G}_\La^{(0)}
  \end{array}
\right.
$$
belongs to $A_R(\mathcal{G}_\La)$. By \cite[Lemma 2.2]{cla15}, we can write
$$f_0=\sum_{D\in F}a_D 1_D$$
where $F$ is a collection of mutually disjoint, nonempty compact open subsets of $\mathcal{G}_\La^{(0)}$. Note that \cite[Proposition 6.3]{cla16} yields that the groupoid $\mathcal{G}_\La$ is effective. Also, for $H:=\mathrm{supp}(f-f_0)$, we have $H\subseteq \mathcal{G}_\La\setminus \mathcal{G}_\La^{(0)}$. Fix $D_0\in F$ with $a_{D_0}\neq 0$. By \cite[Lemma 3.1]{bro14}, there is a nonempty open set $U\subseteq D_0$ such that $UHU=r^{-1}(U)\cap H\cap s^{-1}(U)=\emptyset$. Since the sets $Z(\la\setminus G)$ form a basis of compact open elements for $\mathcal{G}_\La^{(0)}$, there exists $Z(\la\setminus G)\subseteq U$, where $G$ is a finite non-exhaustive subset of $s(\la)\La$. Hence, for every $t\in \mathcal{G}_\La$, we have
$$\left(1_{Z(\la\setminus G)}\ast (f-f_0)\ast 1_{Z(\la\setminus G)}\right)(t)=1_{Z(\la\setminus G)}(r(t)) (f-f_0)(t) 1_{Z(\la\setminus G)}(s(t))=0.$$
Thus the linearity of convolution yields that
$$1_{Z(\la\setminus G)}\ast f\ast 1_{Z(\la\setminus G)}=1_{Z(\la\setminus G)}\ast f_0 \ast 1_{Z(\la\setminus G)}= 1_{Z(\la\setminus G)}.$$

On the other hand, because $G$ is not exhaustive, there exists $\mu\in s(\la)\La$ such that $\mathrm{Ext}(\mu;G)=\emptyset$. Then, using (KP3), we have
$$s_{\mu^*}s_\nu=\sum_{(\rho,\tau)\in\La^{\mathrm{min}}(\mu,\nu)} s_\rho s_{\tau^*}=0$$
for every $\nu\in G$, and hence
\begin{align*}
s_{\mu^*}\left(\prod_{\nu\in G} s_{s(\la)}-s_\nu s_{\nu^*}\right)s_\mu&=s_{\mu^*}\left(s_{s(\la)}+\sum_{F\subseteq G}(-1)^{|F|}\left(\prod_{\nu\in F} s_\nu s_{\nu^*}\right)\right)s_\mu\\
&=s_{\mu^*}s_{s(\la)}s_\mu=s_{s(\mu)}.
\end{align*}
Recall that the isomorphism $\pi$ maps $s_\la\left(\prod_{\nu\in G} (s_{s(\la)}-s_\nu s_{\nu^*})\right)s_{\la^*}$ to the element $1_{Z(\la\setminus G)}$. Therefore, with $c:=s_{(\la\mu)^*}\pi^{-1}(1_{Z(\la\setminus G)}\ast 1_{B^{-1}})$ and $d:=\pi^{-1}(1_{Z(\la\setminus G)})s_{\la\mu}$, we conclude $cad=a_{D_0}s_{s(\mu)}$.
\end{proof}

\begin{lem}\label{lem5.3}
Let $\La$ be a finitely aligned $k$-graph, and let $(\mu,\nu)$ be a generalized cycle with an entrance $\tau$. For $x:=s_{\nu^*}s_\mu$ and $x^*:=s_{\mu^*} s_\nu$, we have $x^* x=s_{s(\mu)}$ and $x^* s_\tau=s_{\tau^*} x=0$. Furthermore, if we define $p_i:=x^i s_\tau$ and $p_i^*:=s_{\tau^*} (x^*)^i$ for $i\geq 1$, then $p_i^* p_j=\delta_{i,j}s_{s(\tau)}$.
\end{lem}

\begin{proof}
By Lemma \ref{lem4.3}, we have $s_\mu s_{\mu^*}\leq s_\nu s_{\nu^*}$. So,
$$x^* x=s_{s(\mu)} x^* x=(s_{\mu^*}s_\mu)(s_{\mu^*}s_\nu)(s_{\nu^*}s_\mu)=s_{\mu^*}(s_\mu s_{\mu^*})s_\mu=s_{s(\mu)}.$$
Also, (KP3) implies
$$x^* s_\tau=(s_{\mu^*} s_\nu)s_\tau =s_{\mu^*}s_{\nu\tau}=\sum_{(\alpha,\beta)\in\La^{\mathrm{min}}(\nu\tau,\mu)}s_\alpha s_{\beta^*}=0$$
because $\La^{\mathrm{min}}(\nu\tau,\mu)=\emptyset$. A same computation shows $s_{\tau^*} x=0$ either.

For the second statement, if $i>j$ we have
$$p_i^* p_j=s_{\tau^*}(x^*)^i x^j s_{\tau}=s_{\tau^*}(x^*)^{i-j}((x^*)^j x^j)s_{\tau}=s_{\tau^*}(x^*)^{i-j-1}(x^*s_{\tau})=0$$
and if $i<j$, then
$$p_i^*p_j=s_{\tau^*} x^{j-i}s_{\tau}=(s_{\tau^*} x)x^{j-i-1} s_{\tau}=0.$$
Finally, for $j=i$, we get
\[p_i^* p_i=s_{\tau^*}(x^*)^i x^i s_{\tau}=s_{\tau^*} s_{s(\mu)} s_{\tau}=s_{s(\tau)}.\]
\end{proof}

We are now ready to prove the main result of article. If $v\in\La^0$ and $(\mu,\nu)$ is a generalized cycle in $\La$, we say that \emph{$v$ is reached from $(\mu,\nu)$} whenever $v\leq s(\mu)$ (i.e., there is a path from $s(\mu)$ to $v$).

\begin{thm}\label{thm6.4}
Let $\La$ be a finitely aligned $k$-graph. If
\begin{enumerate}
  \item $R$ is a field,
  \item $\La$ is aperiodic and cofinal, and
  \item every vertex of $\La$ is reached from a generalized cycle with an entrance,
\end{enumerate}
then $\kp$ is simple and purely infinite.
\end{thm}

\begin{proof}
We use the equivalence (i) $\Longleftrightarrow$ (v) of \cite[Proposition 10]{abr06}: $\kp$ is purely infinite simple if and only if $\kp$ is not a division ring and for every nonzero elements $a,b\in\kp$, there exist $c,d\in\kp$ such that $cad=b$.

Suppose that the three conditions hold. If $|\La^0|\geq 2$ and $v\neq w\in \La^0$, then $s_v s_w=0$ and $s_v,s_w$ are zero divisors. If $|\La^0|=1$, then there exist distinct $\mu,\nu\in \La^{e_i}$ for some $1\leq i\leq k$, because in the otherwise $\kp$ is isomorphic to a nonsimple Laurent polynomial ring $R[x_1,x_1^{-1},\ldots, x_l,x_l^{-1}]$. So, we have $s_{\nu^*}s_\mu=0$ by (KP3) and $s_\mu$ is a zero divisor in $\kp$. Thus, in each case, $\kp$ is not a division ring.

Now fix nonzero elements $a,b\in\kp$. By Lemma \ref{lem5.2}, there are $c',d'\in \kp$ such that $c'ad'=rs_v$ for some $v\in \La^0$ and $r\in R\setminus\{0\}$. Assume $(\mu,\nu)$ is a generalized cycle with an entrance $\tau$ which connects to $v$ by a path $\la\in v\La s(\mu)$. Note that we have
$$s_{\la^*}(c'a d')s_\la=s_{\la^*}(rs_v)s_\la=rs_{s(\la)}=rs_{s(\mu)}.$$
Since $\kp$ is simple by Proposition \ref{prop4.6} (or \cite[Theorem 9.4]{cla16}), the ideal generated by $rs_{s(\tau)}$ is equal to $\kp$. So, there exist $\{c_i,d_i\in\kp:1\leq i\leq l\}$ such that $\sum_{i=1}^l c_i(rs_{s(\tau)})d_i=b$. As Lemma \ref{lem5.3}, set $p_i:=(s_{\nu^*}s_\mu)^i s_\tau$ and $p_i^*:=s_{\tau^*}(s_{\mu^*} s_\nu)^i$ for $i\geq 1$. If we define $c''=\sum_{i=1}^lc_ip_i^*$ and $d''=\sum_{j=1}^lp_j d_j$, Lemma \ref{lem5.3} implies that
\begin{align*}
c''(rs_{s(\mu)})d''&=\left(\sum_{i=1}^lc_ip_i^*\right) rs_{s(\mu)} \left(\sum_{j=1}^lp_j d_j\right)\\
&=\sum_{i,j=1}^l r c_i(p_i^*p_j)d_j\\
&=\sum_{i=1}^l rc_i s_{s(\tau)} d_i\\
&=b.
\end{align*}
Therefore, by setting $c:=c''s_{\la^*} c'$ and $d:=d's_\la d''$, we get $cad=b$. Now \cite[Proposition 10]{abr06} follows the result.
\end{proof}

%%%%%%%%%%%%%%%%%%%%%%%%%%%%%%%%%%%%%%%%%%%%%%%%

\section{A dichotomy principle for simple Kumjian-Pask algebras}\label{sec7}

Despite simple Leavitt path algebras, there exists a simple Kumjian-Pask algebra which is neither purely infinite nor locally matricial \cite[Theorem 7.10]{pin13}. In this section, we consider finitely aligned $k$-graphs $\La$ such that every vertex of $\La$ can be reached only from finitely many vertices. Note that every $k$-graph with finitely many vertices satisfies this condition. In this case, Theorem \ref{thm7.2} below gives a facile necessary and sufficient criterion so that a Kumjian-pask algebra $\kp$ is purely infinite simple. We see also in the case that a simple Kumjian-Pask algebra is either locally matricial or purely infinite.

In order to prove Theorem \ref{thm7.2}, we need the following two lemmas.

\begin{lem}\label{lem7.1}
Let $\La$ be a finitely aligned $k$-graph with the property that the sets $\La^0_{\geq v}:=\{w\in\La^0: v\leq w\}$ are finite for all $v\in \La^0$. Then every cycle of $\La$ is reached from an initial cycle. In particular, if $\La$ is also aperiodic, then every cycle in $\La$ is reached from an initial cycle with an entrance.
\end{lem}

\begin{proof}
Fix a cycle $\mu\in \La$ and let us denote the hereditary set $H:=\La^0_{\geq r(\mu)}=\{w\in \La^0: r(\mu)\leq w\}$. We consider the $k$-subgraph $H\La=(H,r^{-1}(H),r,s,d)$ of $\La$ that contains finitely many vertices. For every cycle $\la$, define $c(\la):=|\{i:d(\la)_i\neq 0\}|$ and choose a cycle $\rho\in H\La$ such that $c(\rho)$ is maximum among those in $H\La$.

We claim that if $\la\in r(\rho)\La$ with $d(\la)\wedge d(\rho)=0$, we then have $\la(m)\neq \la(n)$ for $m<n\leq d(\la)$. For this, assume $\la(m)=\la(n)$ with $m\neq n$. Set $t:=|H|$, $\tau:=\rho^t \la(0,m) (\la(m,n))^t$, and $p:=d(\rho)+(n-m)$. Then the vertices $\tau(0), \tau(p), \ldots, \tau(tp)$ are not distinct (because their number is more than $t=|H|$); so there exists $r<s\leq t$ such that $\tau(rp)=\tau(sp)$. Since for the cycle $\tau(rp,sp)$, we have
$$c(\tau(rp,sp))=c(\rho)+c(\la(m,n)) > c(\rho),$$
this contradicts our choice of $\rho$. Hence, the claim holds.

Now for each $\la\in r(\mu)\La$, let us define the nonnegative integer $b(\la):=\sum_{e_i\wedge d(\rho)=0}d(\la)_i$. As $t=|H|<\infty$, the above claim implies that $\max\{b(\la):\la\in r(\mu)\La\}$ is finite; denote it by $N$. Note that if $\la\in r(\mu)\La$ with $b(\la)=N$, then $s(\la)\La^{e_i}=\emptyset$ for every $e_i\wedge d(\rho)=0$. Select some $\la\in r(\mu)\La$ with $d(\la)=N$ and factorise $\rho^t \la=\alpha\beta$ with $d(\alpha)=d(\la)$ and $d(\beta)=td(\rho)$. Again, since $t=|H|$, the vertices $\beta(0),\beta(d(\rho)),\ldots,\beta(td(\rho))$ are not distinct. So, there exist $r<s\leq t$ such that $\beta(rd(\rho))=\beta(sd(\rho))$. As $r(\beta)\La^{e_i}=s(\alpha)\La^{e_i}=\emptyset$ whenever $e_i\wedge d(\rho)=0$, we see that $\beta(rd(\rho),sd(\rho))$ is an initial cycle. Therefore, $\mu$ is reached from the initial cycle $\beta(rd(\rho),sd(\rho))$, as desired.

Furthermore, if $\La$ is aperiodic, Corollary \ref{cor4.5}(2) implies that every initial cycle has an entrance which follows the second statement.
\end{proof}

\begin{lem}\label{lem6.2}
Let $\Lambda$ be a finitely aligned $k$-graph and $R$ be a unital commutative ring. If $v\in \Lambda^0$ is a vertex such that $v\Lambda=\{v\}$, then $I_v\cong \mathbb{M}_{|\Lambda v|}(R)$ as $R$-algebras, where $I_v$ is the ideal of $\kp$ generated by $s_v$.
\end{lem}

\begin{proof}
For each $(\mu,\nu)\in \Lambda v\times \Lambda v$, define $\theta_{\mu,\nu}:=s_\mu s_{\nu^*}$. Since $v$ receives no nontrivial paths, apply (KP2) and (KP3) to get
\begin{align*}
\theta_{\mu,\nu} \theta_{\lambda,\gamma}&=s_\mu (s_{\nu^*}s_\lambda) s_{\gamma^*}\\
&=s_\mu \left( \sum_{(\alpha,\beta)\in\La^{\mathrm{min}}(\nu,\la)}s_\alpha s_{\beta^*}\right) s_{\gamma^*}\\
&=s_\mu (\delta_{\nu,\lambda} s_v) s_{\gamma^*}\\
&= \delta_{\nu,\lambda} \theta_{\mu,\gamma}
\end{align*}
for every $(\mu,\nu),(\lambda,\gamma)\in \Lambda v\times \Lambda v$. Hence, $\{\theta_{\mu,\nu}:(\mu,\nu)\in \Lambda v\times \Lambda v\}$ forms a set of matrix units indexed by $\La v\times \La v$ which generates a subalgebra in $I_v$ isomorphic to $\mathbb{M}_{|\Lambda v|}(R)$. On the other hand, since $v\Lambda=\{v\}$ is a hereditary subset of $\Lambda^0$, Lemma \ref{lem3.6} implies that the elements $\theta_{\mu,\nu}=s_\mu s_{\nu^*}$ span $I_v$ either. Consequently, $I_v$ is isomorphic to $\mathbb{M}_{|\Lambda v|}(R)$.
\end{proof}

\begin{thm}\label{thm7.2}
Let $\La$ be a finitely aligned $k$-graph such that the sets $\La^0_{\geq v}:=\{w\in\La^0: v\leq w\}$ are finite for all $v\in \La^0$. Then $\kp$ is purely infinite simple if and only if $R$ is a field, and $\La$ is both cofinal and aperiodic containing a cycle.
\end{thm}

\begin{proof}
($\Longrightarrow$): Assume that $\kp$ is simple and purely infinite. Then by Proposition \ref{prop4.6} (or \cite[Theorem 9.4]{cla16}), $R$ is a field and $\La$ is aperiodic and cofinal. By way of contradiction, suppose that $\La$ has no cycles. Since $\La^0_{\geq v}=\{w\in \La^0:v\leq w\}$ is finite for every $v\in \La^0$, there is $w\in \La^0$ such that $w\La^0=\{w\}$. If $I_w$ is the ideal of $\kp$ generated by $s_w$, the simplicity yields that $\kp=I_w$. But we have $I_w\cong M_{|\La w|}(R)$ by Lemma \ref{lem6.2} which is not purely infinite. This contradicts the hypothesis.

($\Longleftarrow$): Conversely, assume that $R$ is a field and $\La$ is aperiodic and cofinal containing a cycle. We first show that every vertex of $\La^0$ is reached from a cycle. Fixed $v\in \La^0$, consider the hereditary subset $H:=\{w\in \La^0: v\leq w\}$. Using Theorem \ref{thm3.8}, the cofinality yields $\overline{H}=\La^0$. Take some cycle $\mu$ in $\La$. Since $r(\mu)\in \overline{H}$, the saturation property of $\overline{H}$ gives $E\in r(\mu)\mathrm{FE}(\La)$ such that $s(E)\subseteq H$. For each $\nu\in \La$ denote $|\nu|:=d(\nu)_1+\ldots+ d(\nu)_k$ and select $\la\in E$ such that $|\la|=\max\{|\nu|:\nu\in E\}$. Since $E$ is exhaustive, the set $\mathrm{Ext}(\mu^t\la;E)$ is nonempty for $t:=|H|$ the cardinality of $H$. So, there exist $\beta\in \mathrm{Ext}(\mu^t\la;E)$, $\nu\in E$ and $\alpha\in s(\nu)\La$ such that $\mu^t\la\beta=\nu\alpha$. In particular, $r(\alpha)=s(\nu)\in H$ and we have
$$|\alpha|=|\mu^t|+|\beta|+(|\la|-|\nu|)\geq |\mu^t|+|\beta|\geq |H|.$$
Therefore, there is a cycle as a submorphism of $\alpha$ that connects to $v$.

Now since $\La$ is aperiodic, combine the above argument with Lemma \ref{lem7.1} to see that every vertex of $\La$ is reached from an initial cycle with an entrance. Hence, by Theorem \ref{thm6.4}, $\kp$ is purely infinite and simple.
\end{proof}

Using Theorem \ref{thm7.2}, we obtain the analogue of \cite[Corollary 5.7]{eva12} for simple Kumjian-Pask algebras.

\begin{cor}[A dichotomy principle for simple Kumjian-Pask algebras]
Let $\La$ be a finitely aligned $k$-graph such that $\La^0_{\geq v}$ is finite for every $v\in \La^0$. Suppose also that $\kp$ is simple. If $\La$ has no cycles, then $\kp$ is locally matricial; otherwise, $\kp$ is purely infinite simple.
\end{cor}

\begin{proof}
Recall from Proposition \ref{prop4.6} that $\La$ is aperiodic and cofinal. If $\La$ has no cycles, then by Lemma \ref{lem6.2} and the first paragraph of above proof we have $\kp\cong \mathbb{M}_{|\La v|}(R)$ for some $v\in\La^0$. If $\La$ contains a cycle, Theorem \ref{thm7.2} implies that $\kp$ is purely infinite.
\end{proof}

\begin{rem}
If $\La^0$ is finite, the set $\La^0_{\geq v}=\{w\in\La^0:v\leq w\}$ is finite for all $v\in \La^0$. Since a Kumjian-Pask algebra $\kp$ is unital if and only if $\La^0$ is finite, the above corollary covers all unital simple Kumjian-Pask algebras.
\end{rem}

%%%%%%%%%%%%%%%%%%%%%%%%%%%%%%%%%%%%%%%%%%%%%%%%

\end{document}